\documentclass[11pt,reqno]{amsart} 

\usepackage{amssymb,latexsym}
\usepackage[draft=true]{hyperref}
\usepackage{cite} 

\usepackage[height=190mm,width=130mm]{geometry} 

\theoremstyle{plain}
\newtheorem{theorem}{Theorem}
\newtheorem{lemma}{Lemma}
\newtheorem{corollary}{Corollary}
\newtheorem{proposition}{Proposition}

\theoremstyle{definition}
\newtheorem{definition}{Definition}

\theoremstyle{remark}
\newtheorem{remark}{Remark}
\newtheorem{example}{Example}

\numberwithin{equation}{section} 

\begin{document}
	\title[Certain results on Kenmotsu pseudo-metric manifolds]{Certain results on Kenmotsu pseudo-metric manifolds} 
	
	\author{Devaraja Mallesha Naik}
	\address{Department of Mathematics\\Kuvempu University \\ Shankaraghatta\\ 577-451
		Shivamogga\\ India}
	
	\email{devarajamaths@gmail.com}
	
	\thanks{The first author (D.M.N.) is grateful to University Grants Commission, New Delhi (Ref. No.:20/12/2015(ii)EU-V) for financial support in the form of Junior Research Fellowship.}
	
	\author{Venkatesha}
	
	\address{Kuvempu University\\ Department of Mathematics\\ Shankaraghatta\\ 577-451
		Shivamogga\\ India}
	
	\email{vensmath@gmail.com}
	
	\author{D.G. Prakasha}
	
	\address{Department of Mathematics\\ Karnatak University\\ Dharwad, Karnataka\\ India}
	
	
	\email{prakashadg@gmail.com}
	
	\begin{abstract}
		In this paper, a systematic study of Kenmotsu pseudo-metric manifolds are introduced. After studying the properties of this manifolds, we provide necessary and sufficient condition for Kenmotsu pseudo-metric manifold to have constant $\varphi$-sectional curvature, and prove the structure theorem for $\xi$-conformally flat and $\varphi$-conformally flat Kenmotsu pseudo-metric manifolds. Next, we consider Ricci solitons on this manifolds. In particular, we prove that an $\eta$-Einstein Kenmotsu pseudo-metric manifold  of dimension higher than 3 admitting a Ricci soliton is Einstein, and  a Kenmotsu pseudo-metric 3-manifold admitting a Ricci soliton is of constant curvature $-\varepsilon$.		
	\end{abstract}
	
	
	
	
	\maketitle
	
\section{Introduction}
The study of contact metric manifolds with associated pseudo-Riemannian metrics were first started by Takahashi \cite{TT} in 1969. Since then, such structures were studied by several authors mainly focusing on the special case of Sasakian manifolds. The case of contact Lorentzian structures $(\eta, g)$, where $\eta$ is a contact one-form and $g$ is a Lorentzian metric associated to it, has a particular relevance for physics and was considered in \cite{Dug} and \cite{BejDug}. A systematic study of almost contact semi-Riemannian  manifolds was undertaken by Calvaruso and Perrone \cite{CalPer1} in 2010, introducing all the technical apparatus which is needed for further investigations.

On the other hand, in 1972 Kenmotsu \cite{Kenm} investigated a class of contact metric manifolds satisfying some special conditions, and after onwards such manifolds are came to known as Kenmotsu manifolds. Recently, Wang and Liu \cite{Wang} investigated almost Kenmotsu manifolds with associated pseudo-Riemannian metric. These are called almost Kenmotsu pseudo-metric manifolds. In this paper, we undertake the systematic study of Kenmotsu pseudo-metric manifolds.

The present paper is organized as follows: In Section~\ref{S:02}, we give the basics of Kenmotsu pseudo-metric manifold $(M,g)$. Certain properties of Kenmotsu pseudo-metric manifolds are provided in Section~\ref{S:03}. We devote Section~\ref{S:04} to the study of curvature properties of Kenmotsu pseudo-metric manifold $(M,g)$ and gave necessary and sufficient condition for $(M,g)$ to have constant $\varphi$-sectional curvature. In Section~\ref{S:05}, we prove necessary and sufficient condition for Kenmotsu pseudo-metric manifold to be $\xi$-conformally flat (and $\varphi$-conformally flat). The last section is focused on the study of Kenmotsu pseudo-metric manifolds whose metric is a Ricci soliton. We show that if $(M,g)$ is a Kenmotsu pseudo-metric manifold admitting a Ricci soliton, then the soliton constant $\lambda=2n\varepsilon$, where $\varepsilon=\pm 1$. Moreover, we show that if $M$ is an $\eta$-Einstein manifold of dimension higher than 3 admitting Ricci soliton, then $M$ is Einstein. Further we show that, a Kenmotsu pseudo-metric manifold $(M,g)$ of dimension 3 admitting Ricci soliton is of constant curvature $-\varepsilon$, where $\varepsilon=\pm 1$. Finally, an illustrative example is constructed which verifies our results.


\section{Preliminaries}\label{S:02}
Let $M$ be a $(2n+1)$ dimensional smooth manifold. We say that $M$ has an \textit{almost contact structure} if there is a tensor field $\varphi$ of type $(1, 1)$, a vector field $\xi$ (called the \textit{characteristic vector field} or \textit{Reeb vector field}), and a 1-form $\eta$ such that
\begin{equation} \label{E:e1}
\varphi^2 = -I + \eta\otimes\xi, \quad \eta (\xi) = 1, \quad
\varphi\xi=0, \quad \eta\circ\varphi=0.
\end{equation}

If $M$ with $(\varphi, \xi, \eta)$-structure is endowed with a
pseudo-Riemannian metric $g$ such that
\begin{equation}\label{E:2.2}
g(\varphi X, \varphi Y)=g( X, Y)-\varepsilon\eta(X)\eta(Y),
\end{equation}
where $\varepsilon=\pm 1$, for all $X,Y\in TM$, then $M$ is called an \textit{almost contact pseudo-metric	manifold}. The relation \eqref{E:2.2} is equivalent to
\begin{equation}\label{E:e3a}
\eta(X)=\varepsilon g(X, \xi)\; {\rm along} \; {\rm with}\;
g(\varphi X,Y)=-g(X,\varphi Y).
\end{equation}
In particular, in an almost contact pseudo-metric manifold, it
follows that $g(\xi, \xi)=\varepsilon$ and so, the characteristic
vector field $\xi$ is a unit vector field, which is either
space-like or time-like, but cannot be light-like.

\par The \textit{fundamental $2$-form} of an almost contact
pseudo-metric manifold is defined by
\[
\Phi(X,Y)=g(X, \varphi Y),
\]
which satisfies $\eta \wedge \Phi^{n}\neq 0$. An almost contact
pseudo-metric manifold is said to be a \textit{contact pseudo-metric
	manifold} if $d\eta=\Phi$. The Riemannian curvature tensor $R$ is given by $R(X, Y)=[\nabla_X, \nabla_Y]-\nabla_{[X, Y]},$ which is opposite to the one used in
\cite{CalPer1}. The Ricci operator $Q$ is
determined by
\[
Ric(X, Y ) = g(QX, Y ),
\]
where $Ric$ denotes the Ricci tensor. In an almost contact pseudo-metric manifold  there always exists a special kind of local pseudo-orthonormal
basis $\{e_i, \varphi e_i, \xi\}_{i=1}^n$, called a local
pseudo $\varphi$-basis.

Consider the manifold $M\times \mathbb{R}$, where $M$ is an almost contact pseudo-metric manifold. Denoting the vector field on $M\times \mathbb{R}$ by $(X, f\frac{d}{dt})$, where $X\in TM$, $t\in \mathbb{R}$, and
$f$ is a smooth function $M\times \mathbb{R}$, we define the structure $J$
on $M\times \mathbb{R}$ by
\[
J\left(X, f\frac{d}{dt}\right)=\left(\varphi X - f\xi,
\eta(X)\frac{d}{dt}\right),
\]
which defines an almost complex structure. If $J$ is integrable, we say that the almost contact pseudo-metric structure $(\varphi, \xi, \eta)$ is normal. Necessary and sufficient condition for integrability of $J$ is \cite{CalPer1}
\[
[\varphi, \varphi](X, Y)+2d\eta(X, Y)\xi=0.
\] 

The following can be easily obtained.
\begin{proposition}
	An almost contact pseudo-metric manifold is normal if and only
	if
	\begin{equation}
	(\nabla_{\varphi X} \varphi )Y - \varphi (\nabla_X\varphi
	)Y+(\nabla_X\eta )(Y)\xi=0,\label{E:2.3}
	\end{equation}
	where $\nabla$ is the Levi-Civita connection.
\end{proposition}

An \textit{almost Kenmotsu pseudo-metric manifold} is an almost contact pseudo-metric manifold with $d \eta=0$ and $d \Phi=2 \eta\wedge \Phi$. A normal almost Kenmotsu pseudo-metric manifold is called a Kenmotsu pseudo-metric manifold \cite{Wang}. Equivalently, from \eqref{E:2.3} we have the following:
\begin{definition}
	Almost contact pseudo-metric manifold is said to be \textit{Kenmotsu pseudo-metric manifold} if 
	\begin{equation}\label{E:2.4}
	(\nabla_X \varphi)Y=-\eta(Y)\varphi X-\varepsilon g(X, \varphi Y)\xi.
	\end{equation}
\end{definition}

From \eqref{E:2.4}, we see
\begin{equation}\label{E:2.5}
\nabla \xi=I-\eta\otimes\xi.
\end{equation}
A straight forward calculation gives the following:
\begin{proposition}\label{P:2.03}
	On Kenmotsu pseudo-metric manifold $(M,g)$, we have
	\begin{gather}
		(\nabla_X\eta)Y=\varepsilon g(X,Y)-\eta(X)\eta(Y),\label{E:2.6}\\
		\pounds_\xi g=2g-\varepsilon \eta\otimes\eta,\label{E:2.7}\\
		\pounds_\xi \varphi=0,\label{E:2.8}\\
		\pounds_\xi \eta=0,\label{E:2.9}
	\end{gather}
	where $\pounds$ denotes the Lie derivative.
\end{proposition}

\section{Some properties of Kenmotsu pseudo metric manifolds}\label{S:03}
For $X\in \text{Ker } \eta$, either space-like or time-like, the \textit{$\xi$-sectional curvature} $K(\xi, X )$ and \textit{$\varphi$-sectional curvature} $K(X, \varphi X)$ are defined respectively
as
\begin{align*}
	K(\xi, X)&=\frac{g(R(\xi,X)X, \xi)}{\varepsilon g(X,X)},\\
	K(X, \varphi X)&=\frac{g(R(\varphi X,X)X, \varphi X)}{ g(X,X)^2}.
\end{align*}
Now we prove:
\begin{proposition}
	If $(M,g)$ is a Kenmotsu pseudo-metric manifold, then we have
	\begin{align}
		R(X, Y)\xi&=\eta(X)Y-\eta(Y)X,\label{E:3.1}\\
		\eta(R(X,Y)Z)&=\eta(Y)g(X,Z)-\eta(X)g(Y, Z),\label{E:3.02}\\
		R(X,\xi)Y&=\varepsilon g(X,Y)\xi-\eta(Y)X,\label{E:3.2}\\
		Ric(X,\xi)&=-2n\eta(X) \quad (\Rightarrow Q\xi=-2n\varepsilon \xi),\label{E:3.3}\\
		K(\xi, \cdot)&=-\varepsilon,\label{E:3.4}\\
		(\nabla_Z R)(X,Y,\xi)&=\varepsilon \{g(X,Z)Y-g(Y,Z)X\}-R(X,Y)Z. \label{E:3.5}
	\end{align}
\end{proposition}
\begin{proof}
	Equations \eqref{E:2.5} and \eqref{E:2.6} give \eqref{E:3.1}. Equations \eqref{E:3.02}, \eqref{E:3.2}, \eqref{E:3.3} and \eqref{E:3.4} are consequences of \eqref{E:3.1}. Equation \eqref{E:3.5} follows from \eqref{E:2.5}, \eqref{E:2.6} and \eqref{E:3.1}.
\end{proof}
\begin{definition}
	An almost contact pseudo-metric manifold for which \[\varphi^2(\nabla_W R)(X, Y,Z)=0,\] for all $X,Y,Z,W\in TM$ is said to be \textit{globally $\varphi$-symmetric}. 
\end{definition}
Using \eqref{E:3.02} and \eqref{E:3.5}, we have the following:
\begin{corollary}
	A globally $\varphi$-symmetric Kenmotsu pseudo-metric manifold is of constant curvature $-\varepsilon$.
\end{corollary}
A Kenmotsu pseudo-metric manifold $M$ is said to be $\eta$-Einstein if the Ricci tensor satisfies
\begin{equation}\label{E:4.1}
Ric(X,Y)=a g(X,Y)+b\eta(X)\eta(Y),
\end{equation}
where $a$ and $b$ are certain smooth functions on $M$. If $b=0$, then $M$ is called an \textit{Einstein} manifold.

From (\ref{E:3.3}), we have
\begin{equation}\label{E:4.2}
\varepsilon a+b=-2n.
\end{equation}
Contracting \eqref{E:4.1} and using \eqref{E:4.2}, we get
\begin{equation}\label{E:4.3}
a=\left( \frac{r}{2n}+\varepsilon\right) , \quad b=-\left( \frac{\varepsilon r}{2n}+2n+1\right),
\end{equation}
where $r$ is the scalar curvature.
Thus, we have: 
\begin{proposition}
	A Kenmotsu pseudo-metric manifold $(M,g)$ is $\eta$-Einstein if and only if 
	\begin{equation}\label{E:4.4}
	Ric(X,Y)=\left( \frac{r}{2n}+\varepsilon\right)g(X,Y)-\left( \frac{\varepsilon r}{2n}+2n+1\right)\eta(X)\eta(Y).
	\end{equation}
\end{proposition}
In particular, we have the following:
\begin{corollary}
	A Kenmotsu pseudo-metric manifold $(M,g)$ is Einstein if and only if 
	\begin{equation}\label{E:4.5}
	Ric(X,Y)=-2n\varepsilon g(X,Y).
	\end{equation}
\end{corollary}
\begin{proposition}
	If the Kenmotsu pseudo-metric manifold $(M,g)$ is $\eta$-Einstein, then
	\begin{equation}\label{E:4.6}
	X(b)+2b\eta(X)=0,
	\end{equation}
	for $n>1$, and for any vector field $X\in TM$.
\end{proposition}
\begin{proof}
	Equation \eqref{E:4.4} is equivalent to
	\begin{equation}\label{E:4.7}
	QY=aY+b\varepsilon \eta(Y)\xi,
	\end{equation}
	where $a$ and $b$ are as in \eqref{E:4.3}. It is well known that 
	\begin{equation}\label{E:4.8}
	\text{div}Q=\frac{1}{2}Dr,
	\end{equation}
	where $D$ denotes the gradient. Equations \eqref{E:4.7} and \eqref{E:4.8} yields to
	\begin{equation*}
		(n-1)Y(a)=\varepsilon \{\xi(b)\eta(Y)+2nb\eta(Y)\}.
	\end{equation*}
	For $Y=\xi$, it gives $\xi(b)=-2b$, and so we get \eqref{E:4.6} for $n>1$.
\end{proof}
\begin{corollary}\label{c:3.7}
	If $b$ (or $a$) is constant in an $\eta$-Einstein Kenmotsu pseudo-metric manifold, then it is Einstein.
\end{corollary}

\section{Curvature properties of Kenmotsu pseudo metric manifolds}\label{S:04}
First we prove the following Lemma which is very useful in subsequent sections.
\begin{lemma}
	On Kenmotsu pseudo-metric manifold $(M,g)$, we have the following identities:
	\begin{align}
		R(X,Y)\varphi Z-\varphi R(X,Y)Z=&\varepsilon \{g(Y,Z)\varphi X-g(X, Z)\varphi Y\nonumber\\	&+g(X,\varphi Z)Y-g(Y, \varphi Z)X\},\label{E:3.6}\\
		R(\varphi X, \varphi Y)Z=R(X,Y)Z&+\varepsilon \{g(Y, Z)X-g(X,Z)Y\nonumber\\
		&+g(Y, \varphi Z)\varphi X-g(X, \varphi Z)\varphi Y\}.\label{E:3.7}
	\end{align}
\end{lemma}
\begin{proof}
	The Ricci identity shows that
	\begin{equation*}
		\nabla_X\nabla_Y\varphi-\nabla_Y\nabla_X\varphi-\nabla_{[X, Y]} \varphi=R(X,Y)\varphi-\varphi R(X,Y).
	\end{equation*}
	Computing the left-hand side using \eqref{E:2.4} yields \eqref{E:3.6}. The equation \eqref{E:3.7} is a result of \eqref{E:3.6}.
\end{proof}
Note that the necessary and sufficient condition for a Sasakian pseudo-metric manifold to have constant $\varphi$-sectional curvature $c$ is  \cite{TT}
\begin{align*}
	4R(X,Y)Z=&(c+3\varepsilon)\{g(Y,Z)X-g(X,Z)Y\}\\
	&+(\varepsilon c-1)\{\eta(X)\eta(Z)Y-\eta(Y)\eta(Z)X\}\\
	&+(c-\varepsilon)\{\eta(Y)g(X,Z)\xi-\eta(X)g(Y,Z)\xi+g(X,\varphi Z)\varphi Y\\
	&-g(Y, \varphi Z)\varphi X+2g(X,\varphi Y)\varphi Z\}.
\end{align*}

Here we prove:
\begin{theorem}
	The necessary and sufficient condition for a Kenmotsu pseudo-metric manifold $M$ to have constant $\varphi$-sectional curvature $c$ is
	\begin{align}
		4R(X,Y)Z=&(c-3\varepsilon)\{g(Y,Z)X-g(X,Z)Y\}\nonumber\\
		&+(c+\varepsilon)\{\varepsilon\eta(X)\eta(Z)Y-\varepsilon\eta(Y)\eta(Z)X\nonumber\\
		&+\eta(Y)g(X,Z)\xi-\eta(X)g(Y,Z)\xi+g(X,\varphi Z)\varphi Y\nonumber\\
		&-g(Y, \varphi Z)\varphi X+2g(X,\varphi Y)\varphi Z\}.\label{E:3.8}
	\end{align}
\end{theorem}
\begin{proof}
	Suppose that $M$ has constant $\varphi$-sectional curvature $c$. Then for all vector fields $U,V\in \text{Ker }\eta$, we have 
	\begin{equation}\label{E:3.9}
	R(U, \varphi U, U, \varphi U)=-cg(U,U)^2.
	\end{equation}
	Using \eqref{E:3.6}, we get
	\begin{align}
		R(U, \varphi V, U, \varphi V)=&R(U, \varphi V, V, \varphi U)+\varepsilon\{g(U,U)g(V,V)\nonumber\\
		&-g(U,V)^2-g(U,\varphi V)^2\},\label{E:3.10}\\
		R(U, \varphi U, V, \varphi U)=&R(U, \varphi U, U, \varphi V),\label{E:3.11}
	\end{align}
	for all $U,V\in \text{Ker }\eta$. Putting $U+V$ in \eqref{E:3.9}, and using\eqref{E:3.7}, \eqref{E:3.10}, \eqref{E:3.11} and the first Bianchi identity, we obtain
	\begin{align*}
		2R(U, \varphi U, U, \varphi V)+2R(V, \varphi V, V, \varphi U)+3R(U, \varphi V, V, \varphi U)
		+R(U,V,U,V)\nonumber\\
		=-c\{2g(U,V)^2+2g(U,U)g(U,V)+2g(U,V)g(V,V)+g(U,U)g(V,V)\}.
	\end{align*}
	Replacing $V$ by $-V$ and then summing the resulting equation to the above equation gives
	\begin{equation}\label{E:3.12}
	3R(U,\varphi V,V,\varphi U)+R(U,V,U,V)=
	-c\{2g(U,V)^2+g(U,U)g(V,V)\}.
	\end{equation}
	Replacing $V$ by $\varphi V$ in \eqref{E:3.12} and then using the identities \eqref{E:3.7} and \eqref{E:3.10}, we get
	\begin{equation}\label{E:3.13}
	4R(U,V,U,V)=(c-3\varepsilon)\{g(U,V)^2-g(U,U)g(V,V)\}-3(c+\varepsilon)g(U,\varphi V)^2.
	\end{equation}
	For $U,V,Z,W\in \text{Ker }\eta$, we determine $R(U+Z,V+W,U+Z,V+W)$ and then using \eqref{E:3.13} we obtain
	\begin{gather}
		4R(U,V,Z,W)+4R(U,W,Z,V)=(c-3\varepsilon)\{g(U,V)g(Z,W)\nonumber\\
		+g(U,W)g(V,Z)-2g(U,Z)g(V,W)\}
		-3(c+\varepsilon)\{g(U,\varphi V)g(Z,\varphi W)\nonumber\\
		+g(U,\varphi W)g(Z, \varphi V)\}.\label{E:3.14}
	\end{gather}
	Interchanging $V$ and $Z$ in \eqref{E:3.14}, and then subtracting the resulting equation with \eqref{E:3.14} and by virtue of the first Bianchi identity we obtain
	\begin{align}
		4&R(U,W,Z,V)=(c-3\varepsilon)\{g(U,V)g(Z,W)-g(U,Z)g(V,W)\}\nonumber\\
		-&(c+\varepsilon)\{g(U,\varphi V)g(Z,\varphi W)-g(U,\varphi Z)g(V,\varphi W)+2g(U,\varphi W)g(Z, \varphi V)\}.\label{E:3.15}
	\end{align}
	Now if $X,Y,Z,W\in TM$, then replacing $U,V,Z,W$ by $\varphi X,\varphi Y,\varphi Z,\varphi W$ in \eqref{E:3.15}, and using \eqref{E:3.6}, \eqref{E:3.7}, and $\eta(R(X,Y)Z)=\eta(Y)g(X,Z)-\eta(X)g(Y, Z)$ we get \eqref{E:3.8}. The converse is trivial.
\end{proof}
\begin{theorem}
	If a Kenmotsu pseudo-metric manifold has constant $\varphi$-sectional curvature $c$, then it is a space of constant curvature and $c=-\varepsilon$.
\end{theorem}
\begin{proof}
	From \eqref{E:3.8}, it is easy to obtain \eqref{E:4.1}, where $a=\frac{1}{2}(n(c-3\varepsilon)+(c+\varepsilon))$ and $b=\frac{-1}{2}\varepsilon(n+1)(c+\varepsilon)$. Since $a$ and $b$ are constants, from Corollary~\ref{c:3.7} it follows that $c=-\varepsilon$.
\end{proof}


\section{Some structure theorems}\label{S:05}
The tangent space $T_pM$ of an almost contact pseudo-metric manifold $M$  can be decomposed as $T_pM=\varphi (T_pM)\oplus L(\xi_p)$, where $L(\xi_p)$ is the linear subspace of $T_pM$ generated by $\xi_p$. Thus the conformal curvature tensor $C$ is defined as a map
\[
C:T_pM \times T_pM \times T_pM \rightarrow \varphi (T_pM)\oplus L(\xi_p), \qquad p\in M,
\]
such that
\begin{align}\label{E:04.01}
	C(X,Y)Z=&R(X,Y)Z-\frac{1}{2n-1}\{Ric(Y, Z)X+g(Y, Z)QX-Ric(X,Z)Y\nonumber \\
	&-g(X, Z)QY\}+\frac{r}{2n(2n-1)}\{g(Y, Z)X-g(X, Z)Y\}.
\end{align}
Then there arise three cases:
\begin{itemize}
	\item The projection of the image of $C$  in $\varphi (T_pM)$ is zero, that is,
	\begin{equation}\label{E:04.02}
	C(X,Y,Z,\varphi W)=0, \qquad \text{ for any }X,Y,Z,W\in T_pM.
	\end{equation}
	\item Projection of the image of $C$ in $L(\xi_p)$ is zero, that is,
	\begin{equation}\label{E:04.03}
	C(X,Y)\xi=0, \qquad \text{ for all }X,Y\in T_pM.
	\end{equation}
	\item Projection of the image of $C\mid_{\varphi(T_pM) \times \varphi(T_pM) \times \varphi(T_pM)}$ in $\varphi (T_pM)$ is zero, that is,
	\begin{equation}\label{E:04.04}
	\varphi^2C(\varphi X, \varphi Y)\varphi Z=0, \qquad \text{ for all }X,Y,Z\in T_pM.
	\end{equation}
\end{itemize}
An almost contact pseudo-metric manifold satisfying the cases \eqref{E:04.02}, \eqref{E:04.03} and \eqref{E:04.04} are said to be conformally symmetric \cite{Zhen}, $\xi$-conformally
flat \cite{ZhenCab} and $\varphi$-conformally flat \cite{CFFZ}, respectively. 

We begin with the following:
\begin{theorem}\label{T:4.0}
	Let $M$ be a $\xi$-conformally
flat Kenmotsu pseudo-metric manifold of dimension higher than 3. Then the scalar curvature $r$ of $M$ satisfies
	\begin{equation}\label{E:4.01}
	D r=\varepsilon\xi(r)\xi,
	\end{equation}
	where $D$ denotes gradient.
\end{theorem}
\begin{proof}
	Since $M$ is $\xi$-conformally
flat, from \eqref{E:04.03} the equation \eqref{E:04.01} becomes
	\begin{align}\label{E:4.02}
		R(U, V)\xi=&\frac{1}{2n-1}\{Ric(V, \xi)U+\varepsilon\eta(V)QU-Ric(U,
		\xi)V-\varepsilon\eta(U)QV\}\nonumber\\
		&-\frac{\varepsilon r}{2n(2n-1)}\{\eta(V)U-\eta(U)V\},
	\end{align}
	for any $U,V\in TM$, and this further gives
	\begin{align}\label{E:4.03}
		R(U, \xi)V=&\frac{1}{2n-1}\{g(V, Q\xi)U+\varepsilon\eta(V)QU-g(QU,V)\xi-g(U,V)Q\xi\}\nonumber\\
		&-\frac{r}{2n(2n-1)}\{\varepsilon\eta(V)U-g(U,V)\xi\}.
	\end{align}
	Putting $V=\xi$ in \eqref{E:4.02}, then differentiating it covariently along $W$ and using \eqref{E:4.03}, we get:
	\begin{align*}
		(\nabla_WR)(U,\xi)\xi=&\frac{1}{2n-1}\{g((\nabla_WQ)\xi,\xi)U+\varepsilon(\nabla_WQ)U-g((\nabla_WQ)U,\xi)\xi\\
		&-\varepsilon\eta(U)(\nabla_WQ)\xi\}-\frac{Wr}{2n(2n-1)}\{\varepsilon U-\varepsilon\eta(U)\xi\}.
	\end{align*}
	Taking the inner product of the above equation with $Y$ and contracting with respect to $U$ and $W$ yield
	\begin{align}\label{E:4.04}
		\sum_{i=1}^{2n+1}\varepsilon_i g((\nabla_{e_i}R)(e_i,\xi)\xi,Y)=&\frac{1}{2n-1}\{g((\nabla_YQ)\xi-(\nabla_\xi Q)Y,\xi)\}\nonumber\\
		&+\frac{\varepsilon(2n-2)}{4n(2n-1)}\{Yr-\eta(Y)\xi(r)\},
	\end{align}
	where $\{e_i\}$ is a pseudo-orthonormal basis in $M$ and  $\varepsilon_i=g(e_i,e_i)$. From the second Bianchi identity we easily obtain
	\begin{equation}\label{E:4.05}
	\sum_{i=1}^{2n+1}\varepsilon_i g((\nabla_{e_i} R)(Y,\xi)\xi,e_i)=g((\nabla_Y Q)\xi-(\nabla_\xi Q)Y,\xi).
	\end{equation} 
	Then from \eqref{E:4.04} and \eqref{E:4.05}, noting that $n>1$ we get
	\begin{equation*}
		g((\nabla_Y Q)\xi-(\nabla_\xi Q)Y,\xi)=\frac{\varepsilon}{4n}\{Yr-\eta(Y)\xi(r)\}.
	\end{equation*}
	Since $\nabla Q$ is symmetric, the above equation becomes
	\begin{equation}\label{E:4.06}
	g((\nabla_Y Q)\xi,\xi)-g((\nabla_\xi Q)\xi,Y)=\frac{\varepsilon}{4n}\{Yr-\eta(Y)\xi(r)\}.
	\end{equation}
	From \eqref{E:3.3}, the left hand side of above equation vanishes. Then \eqref{E:4.06} leads to $Yr=\eta(Y)\xi(r)$ which gives \eqref{E:4.01}.
\end{proof}
\begin{theorem}\label{T:4.1}
	A Kenmotsu pseudo-metric manifold $M$ is $\xi$-conformally
flat if and only if it is an $\eta$-Einstein manifold.
\end{theorem}
\begin{proof}
	If $M$ is $\xi$-conformally
flat, then
	\begin{align*}
		R(X,\xi)\xi=&\frac{1}{2n-1}\{Ric(\xi, \xi)X+\varepsilon QX-Ric(X,
		\xi)\xi-\varepsilon \eta(X)Q\xi\}\\
		&-\frac{\varepsilon r}{2n(2n-1)}\{X-\eta(X)\xi\}.
	\end{align*}
	Making use of equations \eqref{E:3.1} and \eqref{E:3.3} in above gives 
	\begin{equation*}
		Q=\left( \frac{r}{2n}+\varepsilon\right)I-\left( \frac{\varepsilon r}{2n}+2n+1\right)\varepsilon \eta\otimes\xi,
	\end{equation*}
	which is equivalent to \eqref{E:4.4}.
	
	Conversely, suppose that $M$ is $\eta$-Einstein. Formula \eqref{E:04.01} gives
	\begin{align*}
		C(X,Y)\xi=&R(X,Y)\xi-\frac{1}{2n-1}\{Ric(Y, \xi)X+\varepsilon\eta(Y)QX-Ric(X,
		\xi)Y\nonumber \\
		&-\varepsilon\eta(X)QY\} +\frac{\varepsilon r}{2n(2n-1)}\{\eta(Y)X-\eta(X)Y\}.
	\end{align*}
	Now using identities \eqref{E:3.1}, \eqref{E:3.3} and \eqref{E:4.7} results in
	\begin{align*}
		C(X,Y)\xi&=R(X,Y)\xi-\frac{1}{2n-1}\left\{ (2n-\varepsilon a)+\frac{\varepsilon r}{2n}\right\}(\eta(X)Y-\eta(Y)X)\\
		&=R(X,Y)\xi-(\eta(X)Y-\eta(Y)X)=0,
	\end{align*}
	and this concludes the proof.
\end{proof}

\begin{theorem}
	A Kenmotsu pseudo-metric manifold of dimension higher than 3 is  $\varphi$-conformally
flat if and only if it is a space of constant cuvature $-\varepsilon$.
\end{theorem}
\begin{proof}
	Note that the $\varphi$-conformally
flat condition $\varphi^2C(\varphi X, \varphi Y)\varphi Z=0$ is equivalent to $C(\varphi X, \varphi Y, \varphi Z, \varphi W)=0$, and so from \eqref{E:04.01} we get
	\begin{align}\label{E:04.05}
		R(\varphi X&, \varphi Y, \varphi Z, \varphi W)=\frac{1}{2n-1}\{Ric(\varphi Y, \varphi Z)g(\varphi X, \varphi W)+g(\varphi Y, \varphi Z)Ric(\varphi X, \varphi W)\nonumber\\
		&-Ric(\varphi X, \varphi Z)g(\varphi Y, \varphi W)-g(\varphi X, \varphi Z)Ric(\varphi Y, \varphi W)\}\nonumber\\
		&-\frac{r}{2n(2n-1)}\{g(\varphi Y, \varphi Z)g(\varphi X, \varphi W)-g(\varphi X, \varphi Z)g(\varphi Y, \varphi W)\}.
	\end{align} 
	Let $\{E_i=e_i, E_{n+i}=\varphi e_i,E_{2n+1}=\xi\}_{i=1}^{n}$ be a local pseudo-orthonormal $\varphi$-basis. Taking $X=W=E_i$ in \eqref{E:04.05} and summing, we get
	\begin{align}\label{E:04.06}
		\sum_{i=1}^{2n} \varepsilon_i R(&\varphi E_i, \varphi Y, \varphi Z, \varphi E_i)\nonumber\\
		=&\sum_{i=1}^{2n}\varepsilon_i\left[\frac{1}{2n-1} \{Ric(\varphi Y, \varphi Z)g(\varphi E_i, \varphi E_i)+g(\varphi Y, \varphi Z)Ric(\varphi E_i, \varphi E_i)\right. \nonumber\\
		&-Ric(\varphi E_i, \varphi Z)g(\varphi Y, \varphi E_i)-g(\varphi E_i, \varphi Z)Ric(\varphi Y, \varphi E_i)\}\nonumber\\
		&\left. -\frac{r}{2n(2n-1)}\{g(\varphi Y, \varphi Z)g(\varphi E_i, \varphi E_i)-g(\varphi E_i, \varphi Z)g(\varphi Y, \varphi E_i)\}\right] \nonumber\\
		=& \left( \frac{2n-2}{2n-1}\right)  Ric(\varphi Y, \varphi Z)+\frac{1}{2n-1}\left(\frac{r}{2n}+\varepsilon 2n \right)g(\varphi Y, \varphi Z), 
	\end{align}
	where $\varepsilon_i=g(E_i,E_i)$. It can be easily verified that
	\begin{align*}
		\sum_{i=1}^{2n} \varepsilon_i R(\varphi E_i, \varphi Y, \varphi Z, \varphi E_i)&=Ric(\varphi Y, \varphi Z)-\varepsilon R(\xi, \varphi Y, \varphi Z, \xi)\\
		&=Ric(\varphi Y, \varphi Z)+\varepsilon g(\varphi Y, \varphi Z).
	\end{align*}
	So that equation \eqref{E:04.06} becomes
	\begin{equation*}
		Ric(\varphi Y, \varphi Z)=\left(\varepsilon+\frac{r}{2n} \right) g(\varphi Y, \varphi Z).
	\end{equation*}
	Substituting this in \eqref{E:04.05}, one obtains 
	\begin{equation}\label{E:04.07}
	R(\varphi X, \varphi Y, \varphi Z, \varphi W)=\frac{r+4n\varepsilon}{2n(2n-1)}\{g(\varphi Y, \varphi Z)g(\varphi X, \varphi W)-g(\varphi X, \varphi Z)g(\varphi Y, \varphi W)\}.
	\end{equation}
	From \eqref{E:3.7}, \eqref{E:3.6}, \eqref{E:3.02} and \eqref{E:2.2}, we get
	\begin{align}\label{E:04.08}
		R(\varphi X, \varphi Y, \varphi Z, \varphi W)=&R(X,Y,Z,W)+\eta(Y)\eta(Z)g(X,W)-\eta(X)\eta(Z)g(Y,W)\nonumber\\
		&-\eta(Y)\eta(W)g(X,Z)+\eta(X)\eta(W)g(Y,Z).
	\end{align}
	Now \eqref{E:04.07} and \eqref{E:04.08} imply 
	\begin{align}\label{E:04.09}
		R(X,Y,Z,W)=&\frac{r+4n\varepsilon}{2n(2n-1)}\{g(\varphi Y, \varphi Z)g(\varphi X, \varphi W)-g(\varphi X, \varphi Z)g(\varphi Y, \varphi W)\}\nonumber\\
		&-\eta(Y)\eta(Z)g(X,W)+\eta(X)\eta(Z)g(Y,W)\nonumber\\
		&+\eta(Y)\eta(W)g(X,Z)-\eta(X)\eta(W)g(Y,Z).
	\end{align}
	Now taking the scalar product of \eqref{E:3.6} with $W$ and by virtue of \eqref{E:04.09} we get an equation and then contracting the resulting equation with respect to $X$ and $W$ gives
	\begin{equation*}
		(2n-2)\left(\frac{r+4n\varepsilon}{2n(2n-1)}+\varepsilon \right) g(Y,\varphi Z)=0.
	\end{equation*}
	Since $n>1$, it follows that
	\begin{equation}\label{E:04.10}
	r=-\varepsilon2n(2n+1).
	\end{equation}
	Using \eqref{E:04.10} and \eqref{E:2.2} in \eqref{E:04.09}, we get
	\begin{equation*}
		R(X,Y,Z,W)=-\varepsilon \{g(Y,Z)g(X, W)-g(X, Z)g(Y, W)\},
	\end{equation*}
	and so that the manifold is of constant curvature $-\varepsilon$.
	\par The converse is trivial.
\end{proof}

\begin{corollary}
	A conformally flat Kenmotsu pseudo-metric manifold of dimension higher than 3 is a space of constant cuvature $-\varepsilon$.
\end{corollary}
The above corollary for Riemannian case has been proved in \cite{Kenm}.

Now contracting \eqref{E:04.09}, we obtain \eqref{E:4.4}. Thus we can state the following:
\begin{corollary}\label{C:4.4}
	A $\varphi$-conformally
flat Kenmotsu pseudo-metric manifold is an $\eta$-Einstein manifold.
\end{corollary}
In view of Theorem~\ref{T:4.1} and Corrollary~\ref{C:4.4}, we have the following:
\begin{corollary}
	A $\varphi$-conformally
flat Kenmotsu pseudo-metric manifold is always  $\xi$-conformally
flat.
\end{corollary}

\section{Ricci soliton on Kenmotsu pseudo-metric manifolds}\label{S:06}
A \textit{Ricci soliton} on a pseudo-Riemannian manifold $(M,g)$ is defined by 
\begin{equation}\label{E:00}
	(\pounds_V g)(X,Y)+2 Ric(X,Y)+2\lambda g(X,Y)=0,
\end{equation}
where $\lambda$ is a constant. Ricci soliton is a natural generalization of the Einstein metric (that is, $Ric(X,Y)=ag(X,Y)$, for some constant $a$), and  is a special self similar solution of  Hamilton's Ricci flow (see \cite{Ham}) $\frac{\partial}{\partial t}g(t)=-2Ric(t)$ with initial condition $g(0)=g$. We say that the Ricci soliton is \textit{steady} when $\lambda=0$, \textit{expanding}
 when $\lambda>0$ and \textit{shrinking}  when $\lambda <0$.

Before producing the main results, we prove the following:
\begin{lemma}
	A Kenmotsu pseudo-metric manifold $(M,g)$ satisfies 
	\begin{gather}
		(\nabla_X Q)\xi=-QX-2n\varepsilon X,\label{Eq:11}\\
		(\nabla_\xi Q)X=-2QX-4n\varepsilon X.\label{Eq:12}
	\end{gather} 
\end{lemma}
\begin{proof}
	Differentiating $Q\xi=-2n\varepsilon\xi$, and recalling \eqref{E:2.5} provide \eqref{Eq:11}. 
	
	Now differentiating \eqref{E:3.1} along $W$ leads to
	\begin{equation*}
		(\nabla_W R)(X,Y)\xi=-R(X,Y)W+\varepsilon g(X,W)Y-\varepsilon g(Y,W)X.
	\end{equation*}
	Contracting this with respect to $X$ and $W$ gives us
	\begin{equation}\label{Eq:13}
		\sum_{i=1}^{2n+1}\varepsilon_i g((\nabla_{e_i} R)(e_i,Y)\xi, Z)=Ric(Y,Z)+2ng(Y,Z).
	\end{equation}
	From the second Bianchi identity, one can easily obtain	
	\begin{equation}\label{Eq:14}
		\sum_{i=1}^{2n+1}\varepsilon_i g((\nabla_{e_i}R)(Z,\xi)Y,e_i)=g((\nabla_Z Q)\xi,Y)-g((\nabla_\xi Q)Z,Y).
	\end{equation}
	Fetching \eqref{Eq:14} into \eqref{Eq:13} and with the aid of \eqref{Eq:11}, we infer that
	\begin{equation*}
		g((\nabla_\xi Q)Z,Y)=-2Ric(Y,Z)-4ng(Y,Z),
	\end{equation*}
	which proves \eqref{Eq:12}.
\end{proof}

\begin{theorem}
	Let $(M,g)$ be a Kenmotsu pseudo-metric manifold. If $(g,V)$ is a Ricci soliton, then the soliton constant $\lambda=2n\varepsilon$, and so the soliton is either expanding or shrinking depending on the casual character of the Reeb vector field $\xi$.
\end{theorem}
\begin{proof}
	Differentiating \eqref{E:00} covariantly along $Z$ gives
	\begin{align}\label{E:11}
		(\nabla_Z \pounds_V g)(X,Y)=-2(\nabla_Z Ric)(X,Y).
	\end{align}	
	From Yano \cite{Yano}, we know the following well known commutation formula:
	\begin{align*}
		(\pounds_V \nabla_X g-\nabla_X \pounds_V g&- \nabla_{[V,X]} g)(Y, Z)\\
		&=-g((\pounds_V \nabla)(X, Y), Z)-g((\pounds_V \nabla)(X, Z), Y),
	\end{align*}
	for all $X,Y, Z\in TM$. Since $\nabla g=0$, the previous equation gives
	\begin{equation}\label{E:12}
		(\nabla_X \pounds_V g)(Y, Z)=g((\pounds_V \nabla)(X, Y), Z)+g((\pounds_V \nabla)(X, Z), Y),
	\end{equation}
	for all $X,Y, Z\in TM$. As $\pounds_V \nabla$ is a symmetric, it follows from \eqref{E:12} that
	\begin{align}\label{E:13}
		&g((\pounds_V \nabla)(X, Y), Z)\nonumber\\
		&=\frac{1}{2}(\nabla_X \pounds_V g)(Y, Z)+\frac{1}{2}(\nabla_Y \pounds_V g)(Z, X)-\frac{1}{2}(\nabla_Z \pounds_V g)(X, Y).
	\end{align}
	Making use of \eqref{E:11} in \eqref{E:13} we have
	\begin{align}\label{E:14}
		g((\pounds_V \nabla)(X,Y), Z)=(\nabla_Z Ric)(X,Y)-(\nabla_X Ric)(Y,Z)-(\nabla_Y Ric)(Z,X).
	\end{align}
	Putting $Y=\xi$ in \eqref{E:14} and using \eqref{Eq:11} and \eqref{Eq:12}, we obtain
	\begin{equation*}
		(\pounds_V \nabla)(X,\xi)=2QX+4n\varepsilon X.
	\end{equation*}
	Differentiating the preceding equation along $Y$ and using \eqref{E:2.5}, we obtain
	\begin{equation*}
		(\nabla_Y \pounds_V \nabla)(X,\xi)=-(\pounds_V \nabla)(X,Y)+2\eta(Y)\{QX+2n\varepsilon X \}+2(\nabla_Y Q)X.
	\end{equation*}
	Feeding the above obtained expression into the following relation (see \cite{Yano})
	\begin{equation}\label{E:026}
		(\pounds_V R)(X,Y)Z=(\nabla_X\pounds_V\nabla)(Y,Z)-(\nabla_Y\pounds_V\nabla)(X,Z),
	\end{equation}
	and using the symmetry of $\pounds_V \nabla$, we immediately obtain
	\begin{align}\label{Eq:22}
		(\pounds_V R)(X,Y)\xi=&2\eta(X)\{QY+2n\varepsilon Y \}-2\eta(Y)\{QX+2n\varepsilon X \}\nonumber\\
		&+2\{(\nabla_X Q)Y-(\nabla_Y Q)X \}.
	\end{align}
	Setting $Y=\xi$ in the foregoing equation, we get
	\begin{equation}\label{Eq:23}
		(\pounds_V R)(X,\xi)\xi=0.
	\end{equation}
	Now taking the Lie-derivative of $R(X,\xi)\xi=-X+\eta(X)\xi$ along $V$ gives
	\begin{equation*}
		(\pounds_V R)(X,\xi)\xi-2\eta(\pounds_V \xi)X+\varepsilon g(X,\pounds_V \xi)\xi=(\pounds_V \eta)(X)\xi,
	\end{equation*}
	which by virtue of \eqref{Eq:23} becomes
	\begin{equation}\label{Eq:24}
		(\pounds_V \eta)(X)\xi=-2\eta(\pounds_V \xi)X+\varepsilon g(X,\pounds_V \xi)\xi.
	\end{equation}
	With the help of \eqref{E:3.3}, the equation \eqref{E:00} takes the form
	\begin{equation}\label{Eq:25}
		(\pounds_V g)(X,\xi)=-2\lambda\varepsilon\eta(X)+4n\eta(X).
	\end{equation}
	Changing $X$ to $\xi$ in the aforementioned equation gives
	\begin{equation}\label{Eq:26}
		\eta(\pounds_V \xi)=\lambda-2n\varepsilon.
	\end{equation}
	Now Lie-differentiating $\eta(X)=\varepsilon g(X,\xi)$ yields $(\pounds_V \eta)(X)=\varepsilon (\pounds_V g)(X,\xi)+\varepsilon g(X,\pounds_V \xi)$. Using this and \eqref{Eq:26} in \eqref{Eq:24} provides 
	$(\lambda-2n\varepsilon)(X-\eta(X)\xi)=0.$
	Tracing the previous  equation yield $\lambda=2n\varepsilon$.
\end{proof}
\begin{corollary}
	A Kenmotsu manifold admitting the Ricci soliton is always expanding with $\lambda=2n$.
\end{corollary}
\begin{lemma}
	Let $(M,g)$ be a Kenmotsu pseudo-metric manifold. If $(g,V)$ is a Ricci soliton, then the Ricci tensor satisfies
	\begin{equation}\label{Eq:016}
		(\pounds_V Ric)(X,\xi)=-X(r)+\xi(r)\eta(X).
	\end{equation}
\end{lemma}
\begin{proof}
	Contracting equation \eqref{Eq:22} with respect to $X$ and recalling the well-known formulas 
	\begin{equation*}
		\text{div}Q=\frac{1}{2}Dr \quad \text{and}\quad \text{trace}\nabla Q=Dr,
	\end{equation*}
	we easily obtain
	\begin{equation}\label{Eq:27}
		(\pounds_V Ric)(Y,\xi)=-Y(r)-2\eta(Y)\{r+\varepsilon 2n(2n+1) \}.
	\end{equation}
	Substituting $Y=\xi$, we have $(\pounds_V Ric)(\xi,\xi)=-\xi(r)-2\{r+\varepsilon 2n(2n+1) \}$. On the other hand, contracting \eqref{Eq:23} gives $(\pounds_V Ric)(\xi,\xi)=0.$ Using this in the previous equation leads to 
	\begin{equation}\label{Eq:28}
		\xi(r)=-2(r+\varepsilon 2n(2n+1)).
	\end{equation}
	Hence \eqref{Eq:28} and \eqref{Eq:27} give \eqref{Eq:016}.
\end{proof}
Combining Theorem~\ref{T:4.0} and \ref{T:4.1}, we state the following:
\begin{lemma}
	An $\eta$-Einstein Kenmotsu pseudo-metric manifold $M$ of dimension higher than 3 satisfies
	\begin{equation}\label{Eq:29}
		Dr= \varepsilon \xi(r)\xi.
	\end{equation}
\end{lemma}
Now we prove:
\begin{theorem}\label{T:3.2}
	Let $(M,g)$ be an $\eta$-Einstein Kenmotsu pseudo-metric manifold of dimension higher than 3. If $(g,V)$ is a Ricci soliton, then $M$ is Einstein.
\end{theorem}
\begin{proof}
	Making use of \eqref{Eq:29} in \eqref{Eq:016}, we have $(\pounds_V Ric)(X,\xi)=0$.
	Now, Lie-differentiating the first relation of \eqref{E:3.3} along $V$, using \eqref{E:4.4}, \eqref{Eq:25}, $\lambda=2n\varepsilon$ and $\eta(\pounds_V \xi)=0$, we obtain
	\begin{equation*}
		(r+\varepsilon 2n(2n+1))\pounds_V \xi=0.
	\end{equation*}
	If $r=-\varepsilon 2n(2n+1)$, then \eqref{E:4.4} shows that $M$ is Einstein. 
	
	So we assume $r\neq -\varepsilon2n(2n+1)$ in some open set $\mathcal{O}$ of $M$. Hence $\pounds_V \xi=0$ on $\mathcal{O}$, and so it follows from \eqref{E:2.5} that 
	\begin{equation}\label{E:25}
		\nabla_\xi V=V-\eta(V)\xi.
	\end{equation}
	Clearly, \eqref{Eq:25} shows that $(\pounds_V g)(X,\xi)=0$. This together with \eqref{E:25}, we have
	\begin{equation}\label{E:26}
		g(\nabla_X V,\xi)=-g(\nabla_\xi V, X)=-g(X,V)+\eta(X)\eta(V).
	\end{equation}
	From Duggal and Sharma \cite{Duggal}, we know that
	\begin{equation*}
		(\pounds_V \nabla)(X,Y)=\nabla_X\nabla_Y-\nabla_{\nabla_X Y} V+R(V,X)Y.
	\end{equation*}
	Setting $Y=\xi$ in the above equation and by virtue of \eqref{E:2.5}, \eqref{E:3.1}, \eqref{E:25} and \eqref{E:26}, we have $r=-\varepsilon2n(2n+1)$. This leads to a contradiction as $r\neq -\varepsilon 2n(2n+1)$ on $\mathcal{O}$ and completes the proof.
\end{proof}

Now we consider Kenmotsu pseudo-metric 3-manifolds which admits Ricci solitons.
\begin{theorem}\label{T:3.3}
	Let $(M,g)$ be a Kenmotsu pseudo-metric 3-manifold. If $(g,V)$ is a Ricci soliton, then $M$ is of constant curvature $-\varepsilon$.
\end{theorem}
\begin{proof}
	The Riemannian curvature tenor of pseudo-Riemannian 3-manifold is given by
	\begin{align}\label{E:299}
		R(X,Y)Z=&g(Y,Z)QX-g(X,Z)QY+g(QY,Z)X-g(QX,Z)Y\nonumber\\
		&-\frac{r}{2}\{g(Y,Z)X-g(X,Z)Y \}.
	\end{align}
	Taking $Y=Z=\xi$ in \eqref{E:299} and using \eqref{E:3.1} and \eqref{E:3.3} gives 
	\begin{equation} \label{Eq:40}
		Q=\left( \frac{r}{2}+1\right)I-\left( \frac{r}{2}+3\right)\eta\otimes\xi.
	\end{equation}
	Making use of this in \eqref{Eq:22} gives	\begin{align}\label{E:32}
		(\pounds_V R)(X,Y)\xi=&X(r)\{Y-\eta(Y)\xi\}
		+Y(r)\{-X+\eta(X)\xi \}\nonumber\\
		&-(r+6\varepsilon)\{\eta(Y)X-\eta(X)Y\}.
	\end{align}
	Replacing $Y$ by $\xi$ in the above equation and comparing it with \eqref{Eq:23}, we obtain
	\begin{equation*}
	\{\xi(r)+(r+6\varepsilon )\}\{-X+\eta(X)\xi \}=0.
	\end{equation*}
	Contracting the above equation gives $\xi(r)+(r+6\varepsilon )=0$, and consequently it follows from \eqref{Eq:28} that  $r=-6\varepsilon$. Then from \eqref{Eq:40} we have $QX=-2\varepsilon X$, and substituting this into \eqref{E:299} shows that $M$ is of constant  curvature $-\varepsilon$. 
\end{proof}
\begin{corollary}
	There does not exist a Kenmotsu pseudo-metric manifold $(M,g)$ admitting the Ricci soliton $(g,V=\xi)$.
\end{corollary}
\begin{proof}
	If $V=\xi$, then from \eqref{E:2.7} the Ricci soliton equation \eqref{E:00} would become
	\begin{equation}\label{EQ:49}
	Ric=-(1+\lambda)g+\varepsilon\eta\otimes\eta,
	\end{equation}
	which means $M$ is $\eta$-Einstein. Then due to Theorem~\ref{T:3.2} and \ref{T:3.3}, $M$ must be Einstein, and this will be a contradiction to equation \eqref{EQ:49} as $\varepsilon\neq 0$.
\end{proof}
\begin{remark}
	Clearly, Theorem~\ref{T:3.2} and \ref{T:3.3} are generalizations of the results of Ghosh proved in \cite{Ghosh} and \cite{Ghosh2}. Note that our approach and technique to obtain the result is different to that of Ghosh.
\end{remark}

Now we provide an example of a Kenmotsu pseudo-metric 3-manifold which admits a Ricci soliton and verify our results.
\begin{example}\label{Ex:01}
	Let $M=N \times I$, where $N$ is an open connected subset of $\mathbb{R}^2$ and $I$ is an open interval in $\mathbb{R}$.
	Let $(x,y,z)$ be the Cartesian coordinates in $M$. Define the structure $(\varphi, \xi, \eta, g)$ on $M$ as follows:
	\begin{align*}
		\varphi\left( \frac{\partial}{\partial x}\right)& = \frac{\partial}{\partial y},\quad \varphi\left( \frac{\partial}{\partial y}\right) =-\frac{\partial}{\partial x},\quad \varphi\left( \frac{\partial}{\partial z}\right) =0, \\
		\xi&= \frac{\partial}{\partial z},\quad \eta=dz,\\
		(g_{ij})&=
		\begin{pmatrix}
			e^{2z}\quad & 0\quad & 0\\
			0\quad & e^{2z}\quad & 0\\
			0\quad & 0\quad & \varepsilon
		\end{pmatrix}.
	\end{align*}
	Now from Koszul's formula, the Levi-Civita connection $\nabla$ is given by
	\begin{equation}\label{E:43}
		\begin{aligned}
			\nabla_{\partial_1}\partial_1 &=-\varepsilon e^{2z}\partial_3, &\nabla_{\partial_1}\partial_2&=0, &\nabla_{\partial_1}\partial_3&=\partial_1,\\
			\nabla_{\partial_2}\partial_1 &=0, &\nabla_{\partial_2}\partial_2&=-\varepsilon e^{2z}\partial_3, &\nabla_{\partial_2}\partial_3&=\partial_2,\\
			\nabla_{\partial_3}\partial_1 &=\partial_1, &\nabla_{\partial_3}\partial_2&=\partial_2, &\nabla_{\partial_3}\partial_3&=0,
		\end{aligned}
	\end{equation}
	where $\partial_1=\frac{\partial}{\partial x}, \partial_2=\frac{\partial}{\partial y}$ and $\partial_3=\frac{\partial}{\partial z}$. From \eqref{E:43}, one can easily verify
	\begin{equation}
	(\nabla_{\partial_i} \varphi){\partial_j}=-\eta(\partial_j)\varphi \partial_i-\varepsilon g(\partial_i, \varphi \partial_j)\xi,
	\end{equation}
	for all $i,j=1,2,3$, and so $M$  is a Kenmotsu pseudo-metric manifold with the above $( \varphi, \xi, \eta, g)$ structure.
	
	With the help of \eqref{E:43}, we find that: 
	\begin{equation}\label{E:44}
		\begin{gathered}
			R(\partial_1, \partial_2)\partial_3=R(\partial_2, \partial_3)\partial_1 =R(\partial_1, \partial_3)\partial_2 =0,\\
			R(\partial_1, \partial_3)\partial_1=R(\partial_2, \partial_3)\partial_2=\varepsilon e^{2z}\partial_3,\\
			R(\partial_1, \partial_2)\partial_1=\varepsilon e^{2z}\partial_2, \quad R(\partial_2, \partial_3)\partial_3=-\partial_2,\\
			R(\partial_1, \partial_3)\partial_3=-\partial_1, \quad R(\partial_1, \partial_2)\partial_2=-\varepsilon e^{2z}\partial_1.
		\end{gathered}
	\end{equation}
	Let $e_1=e^{-z}\partial_1, e_2=e^{-z}\partial_2$ and $e_3=\xi=\partial_3$. Clearly, $\{e_1, e_2,e_3\}$ forms an orthonormal $\varphi$-basis of vector fields on $M$. Making use of \eqref{E:44} one can easily show that $M$ is Einstein, that is, $Ric(Y,Z)=-2\varepsilon g(Y,Z)$, for any $Y,Z\in TM$. 
	
	Let us consider the vector field
	\begin{equation}\label{E:45}
	V=\alpha\frac{\partial}{\partial y},
	\end{equation}
	where $\alpha$ is a non-zero constant. Making use of \eqref{E:43} one can easily show that $V$ is Killing with respect to $g$, that is, we have $$(\pounds_V g)(X,Y)=g(\nabla_X V, Y)+g(\nabla_Y V, X)=0,$$ for any $X,Y\in TM$. Hence $g$ is a Ricci soliton, that is, \eqref{E:00} holds true with $V$ as in \eqref{E:45} and $\lambda=2\varepsilon$. Further \eqref{E:44} shows that $$R(X,Y)Z=-\varepsilon\{g(Y,Z)X-g(X,Z)Y\},$$ for any $X,Y\in TM$, which means $M$ is of constant curvature $-\varepsilon$ and so Theorem~\ref{T:3.3} is verified.
\end{example}

\section*{Acknowledgement} The authors would like to thank the reviewer for careful and thorough reading of this manuscript and thankful for helpful suggestions towards the improvement of this paper.

\bibliography{mmnsample}
\bibliographystyle{mmn}

\end{document}